\batchmode
\makeatletter
\def\input@path{{\string"/Users/russw/Documents/Research/mypapers/Coset posets are noncontractible/\string"/}}
\makeatother
\documentclass[12pt,oneside,english]{amsart}
\usepackage[T1]{fontenc}
\usepackage[latin9]{inputenc}
\usepackage{geometry}
\usepackage{babel}
\usepackage{amsthm}
\usepackage{amssymb}
\usepackage[unicode=true,pdfusetitle,
 bookmarks=true,bookmarksnumbered=false,bookmarksopen=false,
 breaklinks=false,pdfborder={0 0 1},backref=false,colorlinks=false]
 {hyperref}
\usepackage{breakurl}

\makeatletter
\numberwithin{equation}{section}
\numberwithin{figure}{section}
\theoremstyle{plain}
\newtheorem{thm}{\protect\theoremname}[section]
  \theoremstyle{plain}
  \newtheorem{cor}[thm]{\protect\corollaryname}
  \theoremstyle{remark}
  \newtheorem{rem}[thm]{\protect\remarkname}
  \theoremstyle{plain}
  \newtheorem{lem}[thm]{\protect\lemmaname}
  \theoremstyle{definition}
  \newtheorem{defn}[thm]{\protect\definitionname}

\usepackage{ae}
\usepackage{aecompl}

\makeatother

  \providecommand{\corollaryname}{Corollary}
  \providecommand{\definitionname}{Definition}
  \providecommand{\lemmaname}{Lemma}
  \providecommand{\remarkname}{Remark}
\providecommand{\theoremname}{Theorem}

\begin{document}
\global\long\def\cosetlat{\mathcal{C}}

\global\long\def\sglat{\mathcal{L}}

\global\long\def\rec{\tilde{\chi}}

\global\long\def\Aut{\operatorname{\mathsf{Aut}}}

\global\long\def\Out{\operatorname{\mathsf{Out}}}

\global\long\def\Inn{\operatorname{\mathsf{Inn}}}

\global\long\def\hol{\operatorname{\mathsf{Hol}}}

\global\long\def\diag{\mathrm{diag}}

\global\long\def\ff{\mathbb{F}}

\global\long\def\zz{\mathbb{Z}}

\global\long\def\normalin{\mathrel{\lhd}}

\global\long\def\normalineq{\mathrel{\unlhd}}

\global\long\def\semidirect{\rtimes}

\global\long\def\comma{{,}}

\global\long\def\join{\mathbin{\ast}}

\global\long\def\homol{\widetilde{H}}

\global\long\def\LL{\mathcal{L}}

\global\long\def\kone{\mathcal{K}_{1}}

\global\long\def\ktwo{\mathcal{K}_{2}}

\global\long\def\kboth{\kone\cup\ktwo}

\title[Coset posets are not contractible]{Order complexes of coset posets of finite groups are not contractible}

\author{John Shareshian and Russ Woodroofe}

\thanks{The first author was supported in part by NSF Grants DMS-0902142
and DMS-1202337.}

\address{Department of Mathematics, Washington University in St.~Louis, St.~Louis,
MO, 63130}

\email{shareshi@math.wustl.edu}

\address{Department of Mathematics \& Statistics, Mississippi State University,
Starkville, MS 39762}

\email{rwoodroofe@math.msstate.edu}
\begin{abstract}
We show that the order complex of the poset of all cosets of all proper
subgroups of a finite group $G$ is never $\ff_{2}$-acyclic and therefore
never contractible. This settles a question of K.~S.~Brown.
\end{abstract}
\maketitle

\section{Introduction}

We settle a question asked by K.~S.~Brown in \cite{Brown:2000}.
For a group $G$, $\cosetlat(G)$ will denote the poset of all cosets
of all proper subgroups of $G$, ordered by inclusion. For a poset
$P$, $\Delta P$ will denote the order complex of $P$. Other terms
used but not defined in this introduction are defined in Section~\ref{sec:Pre}. 
\begin{thm}
\label{thm:Main} If $G$ is a finite group, then $\Delta\cosetlat(G)$
is not $\ff_{2}$-acyclic, and therefore is not contractible.
\end{thm}
With some explicitly stated exceptions, the groups, partially ordered
sets and simplicial complexes considered herein are assumed to be
finite. We assume some familiarity with topological combinatorics
(see for example \cite{Bjorner:1995,Wachs:2007}), along with the
rudiments of algebraic topology (see for example \cite{Hatcher:2002,Munkres:1984})
and group theory (see for example \cite{Aschbacher:2000,Dixon/Mortimer:1996}).

\subsection{History and motivation}

The topology of $\Delta\cosetlat(G)$ was studied by Brown in \cite{Brown:2000}.
More general coset complexes were studied from a somewhat different
point of view by Abels and Holz in \cite{Abels/Holz:1993}. However,
from our perspective (and that of Brown), the story begins with the
work of P.~Hall, who in \cite{Hall:1936} introduced generalized
Möbius inversion in order to enumerate generating sequences. Hall
considered the probability $P_{G}(k)$ that a $k$-tuple $(g_{1},\ldots,g_{k})$
of elements of a group $G$, chosen uniformly with replacement, includes
a generating set for $G$. He showed that 
\[
P_{G}(k)=\sum_{H\leq G}\mu(H,G)[G:H]^{-k},
\]
where $\mu$ is the Möbius function on the subgroup lattice of $G$.
(We mention that Weisner introduced generalized Möbius inversion independently
in \cite{Weisner:1935a}. See \cite[Chapter 3]{Stanley:2012} for
a comprehensive discussion of this theory.) 

Bouc observed that $-P_{G}(-1)$ is the reduced Euler characteristic
$\rec(\Delta\cosetlat(G))$. Indeed, Hall showed in \cite{Hall:1936}
that if $\widehat{P}$ is obtained from $P$ by adding a minimum element
$\hat{0}$ and a maximum element $\hat{1}$, then 
\[
\rec(\Delta P)=\mu_{\widehat{P}}(\hat{0},\hat{1}).
\]
A straightforward computation shows that 
\[
\mu_{\widehat{\cosetlat(G)}}(\hat{0},\hat{1})=-P_{G}(-1).
\]
This led to Brown's work, in which he obtained divisibility results
for $P_{G}(-1)$ using group actions on $\Delta\cosetlat(G)$.

Brown found no group $G$ for which $P_{G}(-1)=0$. As the reduced
Euler characteristic of a contractible complex is zero, the question
of contractibility arises naturally. Previous progress on this question
involved showing that $P_{G}(-1)\neq0$. Gaschütz showed in \cite[Satz 2]{Gaschutz:1959}
that $P_{G}(-1)\neq0$ when $G$ is solvable. (Brown refined this
result by calculating the homotopy type of $\Delta\cosetlat(G)$ for
a solvable group $G$ in \cite[Proposition 11]{Brown:2000}.) Patassini
proved $P_{G}(-1)\neq0$ for many almost simple groups $G$ in \cite{Patassini:2009,Patassini:2011}.
He obtained further results for some groups with minimal normal subgroups
that are products of alternating groups in \cite{Patassini:2013}.
The question of whether $P_{G}(-1)$ is nonzero for all (finite) $G$
remains open. 

Abels and Holz consider in \cite{Abels/Holz:1993} a more general
class of posets. Let $G$ be a (possibly infinite) group and $\mathcal{H}$
be a collection of proper subgroups of $G$ that is closed under intersection.
Abels and Holz study the order complex of the poset $\cosetlat_{\mathcal{H}}(G)$
of all cosets of all subgroups in $\mathcal{H}$. In their Theorem~2.4,
they describe relations between connectivity properties of $\Delta\cosetlat_{\mathcal{H}}(G)$
and the structure of $G$. Our Theorem~\ref{thm:Main} says that
$\Delta\cosetlat_{\mathcal{H}}(G)$ is not infinitely connected when
$G$ is finite and $\mathcal{H}$ contains all proper subgroups of
$G$. In contrast, Ramras in \cite[Remark 2.4]{Ramras:2005} noticed
that $\Delta\cosetlat(G)$ is contractible when $G$ is not finitely
generated.

Some other papers on the topology of $\Delta\cosetlat(G)$ are \cite{Torres-Giese:2012,Woodroofe:2007,Woodroofe:2009b}.

\subsection{A brief description of our proof}

Our proof has three main ingredients, namely, a ``join theorem''
of Brown, the Classification of Finite Simple Groups, and P.~A.~Smith
Theory. 

Brown showed that, given a group $G$ and normal subgroup $N$, there
is a subposet $\cosetlat(G,N)$ of $\cosetlat(G)$ such that $\Delta\cosetlat(G)$
is homotopy equivalent to the join $\Delta\cosetlat(G,N)\join\Delta\cosetlat(G/N)$.
This result allows us to use induction on $|G|$. We complete the
proof by showing that $\Delta\cosetlat(G,N)$ is not $\ff_{2}$-acyclic
when $N$ is a minimal normal subgroup of $G$. Such a subgroup is
a direct product of pairwise isomorphic simple groups. 

In order to show $\Delta\cosetlat(G,N)$ is not $\ff_{2}$-acyclic,
we use Smith Theory and the Classification. For each possible minimal
normal subgroup $N$, we describe a group $E$ such that $E$ acts
on $\cosetlat(G,N)$ with no fixed point. Using results of Smith and
Oliver, we choose $E$ so as to preclude a fixed-point-free action
on an $\ff_{2}$-acyclic complex.


\subsection{Further comments}

\subsubsection{}

Theorem~\ref{thm:Main} stands in clear contrast to other results
on order complexes of posets naturally associated to finite groups.
Consider the poset $\sglat(G)$ of nontrivial proper subgroups of
$G$, ordered by inclusion. The complex $\Delta\sglat(G)$ is contractible
for many groups, including all those with nontrivial Frattini subgroup.

Next, let $p$ be a prime. Consider the subposet $\mathcal{S}_{p}(G)$
of $\mathcal{L}(G)$ consisting of all $p$-subgroups and the subposet
$\mathcal{A}_{p}(G)$ of $\mathcal{S}_{p}(G)$ consisting of all elementary
abelian $p$-subgroups. The order complexes $\Delta\mathcal{S}_{p}(G)$
and $\Delta\mathcal{A}_{p}(G)$ were first studied, respectively,
by Brown in \cite{Brown:1974,Brown:1975} and Quillen in \cite{Quillen:1978}.
These two complexes are homotopy equivalent. They are contractible
when $G$ has a nontrivial normal $p$-subgroup. The converse of this
last statement is a well-known conjecture of Quillen, see \cite[Conjecture 2.9]{Quillen:1978}.

\subsubsection{}

The identity $-P_{G}(-1)=\rec(\Delta\cosetlat(G))$ can be considered
to be an example of the phenomenon of \emph{combinatorial reciprocity}.
Often some objects of interest are counted by evaluating an appropriate
function at positive integers. Combinatorial reciprocity occurs when
evaluation of the same function at negative integers counts some closely
related objects. Combinatorial reciprocity is discussed in \cite{Beck:2012,Stanley:1974a,Stanley:1974b,Stanley:2012}.
At this point we know of no interesting interpretation of $P_{G}(-n)$
for integers $n>1$. More generally, one can evaluate $P_{G}$ at
any complex number $s$. The study of $P_{G}(s)$ as a complex function
was initiated by Boston in \cite{Boston:1996} and Mann in \cite{Mann:1996},
and was continued by various authors. See for example \cite{Damian/Lucchini:2007,Damian/Lucchini:2014,Damian/Lucchini/Morini:2004,Detomi/Lucchini:2003,Patassini:2009,Patassini:2011b,Patassini:2013b,Shareshian:1998}.

\subsubsection{}

In Lemma~\ref{lem:UniversalGeneration}~(3) below, we note that
if $L$ is a finite simple group of Lie type or a sporadic simple
group, then there exists some odd prime $p$ such that $L=\langle P,R\rangle$
whenever $P$ is a Sylow $p$-subgroup of $L$ and $R$ is a Sylow
$2$-subgroup of $L$. This property need not hold when $L$ is an
alternating group. However, one can ask whether for each $n$ there
exist primes $p=p(n)$ and $r=r(n)$ such that $\left\langle P,R\right\rangle =A_{n}$
whenever $P$ is a Sylow $p$-subgroup and $R$ is a Sylow $r$-subgroup
of $A_{n}$. This interesting question remains open. It is related
to a question raised by Dolfi, Guralnick, Herzog and Praeger in \cite[Section 6]{Dolfi/Guralnick/Herzog/Praeger:2012}.
These authors ask whether for each $n$ there exist conjugacy classes
$C,D$ in $A_{n}$ consisting of elements of prime-power order, such
that $\left\langle c,d\right\rangle =A_{n}$ for all $(c,d)\in C\times D$.
A positive answer to their question immediately implies a positive
answer to ours. We will address related questions in a forthcoming
paper.

\subsubsection{}

To our knowledge, the first use of Smith Theory in combinatorics appears
in work of Kahn, Saks and Sturtevant. In the paper \cite{Kahn/Saks/Sturtevant:1984},
these authors use the work of Smith and Oliver mentioned above to
obtain a striking result about computational complexity.

\subsection{Contents of the paper}

In Section~\ref{sec:Pre}, we introduce some basic facts and definitions.
In Section~\ref{sec:MainThm}, we reduce the proof of Theorem~\ref{thm:Main}
to some claims about nonabelian finite simple groups. In the remaining
sections, we use the Classification to prove these claims. In Sections~\ref{sec:Alternating}
and \ref{sec:A7}, we prove the required result for alternating groups.
The group $A_{7}$ requires more care than the other alternating groups.
In Section~\ref{sec:LieSporadic}, we handle sporadic groups and
groups of Lie type.

\section*{Acknowledgements}

We thank Anders Björner, Bob Guralnick, Andrea Lucchini, Richard Lyons,
and Massimiliano Patassini for helpful discussions and correspondence.
The first author was inspired to think about this problem during a
visit to the Università degli Studi di Padova in Spring 2011. He thanks
the mathematics department there, and in particular Andrea Lucchini
and Eloisa Detomi, for their gracious hospitality. The second author
is grateful to the University of Primorska for their hospitality (with
support from the European Social Fund and the Ministry of Education,
Science and Sport of the Republic of Slovenia) during his visit of
Spring 2014.

\section{\label{sec:Pre}Preliminaries}

Here we introduce some basic definitions and facts. A reader who is
familiar with topological combinatorics and group theory can skip
this section safely, and refer to it as necessary.

\subsection{Groups and cosets}

As is standard, we write $K^{g}$ for $g^{-1}Kg$ whenever $K\subseteq G$
and $g\in G$, and write $x^{g}$ for $g^{-1}xg$. Similarly, we write
$K^{\alpha}$ and $x^{\alpha}$ for the images of $K$ and $x$ under
an automorphism $\alpha$ of $G$.

When referring to a coset, we mean a right coset. This causes no loss
of generality, as every coset of every subgroup of $G$ is a right
coset of some subgroup. Indeed, $xH=H^{x^{-1}}x$. It is not hard
to see that every coset is a right coset of a unique subgroup.

\subsection{\label{sec:PreSimplicial}Simplicial complexes}

An \textit{abstract simplicial complex} is a collection $\Delta$
of sets (called \textit{faces}) such that if $S\in\Delta$ and $T\subseteq S$
then $T\in\Delta$. We make no distinction between an abstract simplicial
complex and its geometric realization.

Let $P$ be a (finite) poset. The \textit{order complex} $\Delta P$
is the simplicial complex whose $k$-dimensional faces are the chains
of length $k$ (size $k+1$) in $P$.

If $\Delta$ and $\Gamma$ are simplicial complexes on disjoint vertex
sets, the \textit{join} $\Delta\join\Gamma$ is the complex whose
faces are all sets $S\cup T$ such that $S\in\Delta$ and $T\in\Gamma$. 

Associated to a simplicial complex $\Delta$ and a ring $R$ are the
\textit{reduced simplicial homology groups} $\homol_{i}(\Delta;R)$,
as described (for example) in \cite{Munkres:1984}. A complex $\Delta$
is called $R$-\textit{acyclic} if $\homol_{i}(\Delta;R)=0$ for every
integer $i$. Every contractible complex is $R$-acyclic for all $R$.
Every nonempty $R$-acyclic complex has at least one nonempty face.
Indeed, $\homol_{-1}(\{\emptyset\};R)\cong R$, hence the complex
$\{\emptyset\}$ is not acyclic over any ring $R$. The simplicial
complexes that we consider all contain the empty face $\emptyset$.

\section{\label{sec:MainThm}Proof of Theorem~\ref{thm:Main}}

Here we prove Theorem~\ref{thm:Main}, although we defer the proofs
of some key lemmas on simple groups to later sections. Let us first
collect some main ingredients in the proof.

\subsection{\label{sec:BrownJoin}Brown's Join Theorem for $\Delta\protect\cosetlat(G)$}

Given a normal subgroup $N$ of $G$, we define the \textit{relative
coset poset }\textit{\emph{to be}} 
\[
\cosetlat(G,N):=\{Hx\in\cosetlat(G):HN=G\}.
\]
The next result, due to Brown, is key to our proof.
\begin{thm}[{Brown's Join Theorem \cite[Proposition 10]{Brown:2000}}]
 \label{thm:BrownJoin} If $G$ is a group and $N$ is a normal subgroup
of $G$, then $\Delta\cosetlat(G)$ is homotopy equivalent to $\Delta\cosetlat(G/N)\join\Delta\cosetlat(G,N)$.\end{thm}
\begin{cor}
\label{cor:kun} Let $p$ be a prime. Let $N$ be a normal subgroup
of $G$. The complex $\Delta\cosetlat(G)$ is $\ff_{p}$-acyclic if
and only if at least one of $\Delta\cosetlat(G,N)$ and $\Delta\cosetlat(G/N)$
is $\ff_{p}$-acyclic. \end{cor}
\begin{proof}
The result follows immediately from Theorem~\ref{thm:BrownJoin}
and the Künneth Formula for joins (see for example \cite[(9.12)]{Bjorner:1995}). 
\end{proof}
We see now that Theorem~\ref{thm:Main} follows quickly from the
next result.
\begin{thm}
\label{thm:MainIndStep} If $N$ is a minimal normal subgroup of $G$,
then $\Delta\cosetlat(G,N)$ is not $\ff_{2}$-acyclic. \end{thm}
\begin{proof}[Proof (of Theorem~\ref{thm:Main}, assuming Theorem~\ref{thm:MainIndStep}).]
 We proceed by induction on the order of $G$. If $G=1$, then $\Delta\cosetlat(G)=\{\emptyset\}$.
Now assume $\left|G\right|>1$, and let $N$ be a minimal normal subgroup
of $G$. The complex $\Delta\cosetlat(G/N)$ is not $\ff_{2}$-acyclic
by inductive hypothesis and $\Delta\cosetlat(G,N)$ is not $\ff_{2}$-acyclic
by Theorem~\ref{thm:MainIndStep}. Theorem~\ref{thm:Main} now follows
from Corollary~\ref{cor:kun}.
\end{proof}
It remains to prove Theorem~\ref{thm:MainIndStep}. In the rest of
Section~\ref{sec:MainThm}, we show how to reduce the proof to certain
claims about finite simple groups.

\subsection{\label{sec:SmithTheory}Group actions and Smith theory}

In order to prove Theorem~\ref{thm:MainIndStep}, we use Smith Theory. 

Given a group $E$ acting by automorphisms (order preserving bijections)
on a poset $Q$, we write $Q^{E}$ for the fixed point set 
\[
Q^{E}:=\{q\in Q:q^{g}=q\mbox{ for all }g\in E\}.
\]
The action of $E$ on $Q$ induces a simplicial action of $E$ on
$\Delta Q$. 

Work of Smith in \cite{Smith:1941} and of Oliver in \cite{Oliver:1975}
shows that, given a prime $p$, certain groups cannot act without
fixed points on $\ff_{p}$-acyclic complexes. (A clear summary of
this work appears in \cite[Section 1]{Oliver:1975}.) Applying their
results to actions on order complexes, we obtain immediately the next
result.
\begin{thm}[Smith \cite{Smith:1941}, Oliver \cite{Oliver:1975}]
\label{thm:smith} Let $p$ and $r$ be primes. Let $Q$ be a poset
such that $\Delta Q$ is $\ff_{p}$-acyclic. Let $E$ be a group admitting
a normal series $P\normalineq H\normalineq E$ such that
\begin{enumerate}
\item $P$ is a $p$-group,
\item $H/P$ is cyclic, and 
\item $E/H$ is an $r$-group. 
\end{enumerate}
If $E$ acts on $Q$ by automorphisms, then $Q^{E}\neq\emptyset$. \end{thm}
\begin{rem}
It is not necessary that the primes $p,r$ in Theorem~\ref{thm:smith}
be distinct. 
\end{rem}

\begin{rem}
\label{rem:SmithAppl}If $E\cong P\times K$, where $P$ is a $p$-group
and $K$ is either cyclic or an $r$-group, then $E$ satisfies conditions
(1)-(3) of Theorem~\ref{thm:smith}. The same holds for $E\cong(P\times K)\semidirect R$,
where $P$ is a $p$-group, $R$ is an $r$-group, and $K$ is cyclic.
\end{rem}
We will apply Theorem~\ref{thm:smith} to $\cosetlat(G,N)$. There
are two actions on $\cosetlat(G)$ that we wish to consider. The first
action is that of $G\times G$ by left and right translation. That
is, 
\begin{equation}
Hx\cdot(g,h)=g^{-1}Hxh=H^{g}g^{-1}xh\,\,\mbox{for }(g,h)\in G\times G.\label{eq:action1}
\end{equation}
 The second action is by $\Aut(G)$, where 
\begin{equation}
\left(Hx\right)^{\alpha}=H^{\alpha}x^{\alpha}\,\,\mbox{for }\alpha\in\Aut(G).\label{eq:action2}
\end{equation}
The component-wise action of $\Aut(G)$ on $G\times G$ gives rise
to the semidirect product $A:=(G\times G)\semidirect\Aut(G)$. The
actions described in (\ref{eq:action1}) and (\ref{eq:action2}) combine
to form a well-defined action of $A$ on $\cosetlat(G)$, with $(g,h,\alpha)$
mapping $Hx$ to
\[
\left(Hx\cdot(g,h)\right)^{\alpha}=(g^{-1}Hxh)^{\alpha}=(Hx)^{\alpha}\cdot(g^{\alpha},h^{\alpha}).
\]

\begin{rem}
If $\left|G\right|>1$, then the action of $A$ has a nontrivial kernel
$N$. The quotient $A/N$ is called the \emph{holomorph} of $G$.
The kernel of this action will be of no concern to us.
\end{rem}
In all but one of our arguments, we will use subgroups of $(G\times G)\semidirect\Aut(G)$
that are contained in $G\times G$. These subgroups will be of the
form $P\times K$ with $P,K\leq G$. When we mention an action of
such a subgroup on $\cosetlat(G)$, we always mean that $P$ acts
by left translation and $K$ acts by right translation. 

Suppose $N\normalineq G$. If $HN=G$ then $H^{g}N=G$ for all $g\in G$.
It follows that $P\times K$ acts on $\cosetlat(G,N)$. Note that
$P\times1$ fixes $Hx$ if and only if $P\leq H$ and $1\times K$
fixes $Hx$ if and only if $K^{x^{-1}}\leq H$. The next result follows.
\begin{lem}
\label{lem:LRactionFixed} A subgroup $P\times K$ of $G\times G$
fixes $Hx\in\cosetlat(G)$ if and only if $\langle P,K^{x^{-1}}\rangle\leq H$. 
\end{lem}

\subsection{Minimal normal subgroups}

Along with Smith theory, we use the Classification of Finite Simple
Groups to prove Theorem~\ref{thm:MainIndStep}. Suppose $G$ is nontrivial,
and $N$ is a minimal normal subgroup of $G$. There exist some positive
integer $t$ and some simple group $L$ such that $N$ is isomorphic
with the direct product of $t$ copies of $L$. (See for example \cite[Theorem 4.3A(iii)]{Dixon/Mortimer:1996}.)
In this situation, we abuse notation by writing $N=L^{t}$ and representing
an element of $N$ as a $t$-tuple of elements of $L$.

The case where the simple group $L$ is cyclic of prime order was
already handled by Brown.
\begin{lem}
\label{lem:amns}If $G$ has an abelian minimal normal subgroup $N$,
then $\Delta\cosetlat(G,N)$ is not $\ff_{2}$-acyclic.\end{lem}
\begin{proof}
As noted in \cite[Proposition 9]{Brown:2000}, the poset $\cosetlat(G,N)$
is an antichain of size divisible by $|N|$. Therefore, $\Delta\cosetlat(G,N)$
is not connected if it contains a nonempty face. It follows that one
of $\homol_{-1}(\Delta\cosetlat(G,N);\ff_{2})$ or $\homol_{0}(\Delta\cosetlat(G,N);\ff_{2})$
is nontrivial.
\end{proof}
We turn now to the case where $N=L^{t}$ with $L$ nonabelian simple.
A subgroup $K\leq L$ can be embedded in $N$ diagonally, as follows.
\begin{defn}
Given $N=L^{t}$ and $K\leq L$, we define 
\[
K^{\diag}:=\{(k,\ldots,k):k\in K\}\leq N.
\]

\end{defn}
The next definition is key for finding useful group actions on $\Delta\cosetlat(G,N)$.
\begin{defn}
\label{def:Univpgen}Let $G$ be a group, let $H\leq G$ and let $p$
be a prime. We say that $H$ \textit{universally} $p$-\textit{generates}
$G$ if $\langle H,P\rangle=G$ whenever $P$ is a Sylow $p$-subgroup
of $G$. \end{defn}
\begin{rem}
\label{rem:UnivSylpr} Let $p,r$ be primes. A Sylow $r$-subgroup
of $G$ universally $p$-generates $G$ if and only if every maximal
subgroup of $G$ has index divisible by at least one of $p$ and $r$. 
\end{rem}
The importance of Definition~\ref{def:Univpgen} is apparent from
the following lemma.
\begin{lem}
\label{lem:UnivpgenFixedpts}Let $G$ be a group and let $p$ be a
prime. Let $N\unlhd G$. If $K\leq N$ universally $p$-generates
$N$ and $P$ is any Sylow $p$-subgroup of $N$, then $\cosetlat(G,N)^{P\times K}=\emptyset$. \end{lem}
\begin{proof}
Assume for contradiction that $Hx\in\cosetlat(G,N)^{P\times K}$.
By Lemma~\ref{lem:LRactionFixed}, $H$ contains both $P$ and $K^{x^{-1}}$.
As $K$ universally $p$-generates $N$, so does $K^{x^{-1}}$. Therefore,
$H$ contains $N$. This is impossible, as $H<G$ and $HN=G$.\end{proof}
\begin{lem}
\label{lem:UnivProd} Let $L$ be a simple group, let $p$ be a prime,
and let $t$ be a positive integer. If a proper subgroup $K<L$ universally
$p$-generates $L$, then $K^{\diag}$ universally $p$-generates
$N:=L^{t}$. \end{lem}
\begin{proof}
Let $P$ be a Sylow $p$-subgroup of $N$. By assumption, $N=\prod_{i=1}^{t}L_{i}$
with each $L_{i}\cong L$. It is not hard to see that $P=\prod_{i=1}^{t}(P\cap L_{i})$.
Moreover, $P\cap L_{i}$ is a Sylow $p$-subgroup of $L_{i}$ for
each $i\in[t]$. The standard projection of $\langle P,K^{\diag}\rangle$
onto $L_{i}$ thus contains both $P\cap L_{i}$ and $K$ and is therefore
all of $L_{i}$. It follows now from (the conjugacy part of) Sylow's
Theorem that $\langle P,K^{\diag}\rangle$ contains every Sylow $p$-subgroup
of $L_{i}$. 

As $K<L$ universally $p$-generates $L$, $P\cap L_{i}$ is nontrivial.
As $L$ is simple, it follows that the Sylow $p$-subgroups of $L_{i}$
together generate $L_{i}$. Hence, $L_{i}\leq\langle P,K^{\diag}\rangle$
for each $i\in[t]$. \end{proof}
\begin{cor}
\label{cor:EmptyfixProd} Let $L$ be a simple group, let $p$ be
a prime, and let $t$ be a positive integer. Let $G$ be a group with
normal subgroup $N=L^{t}$, and let $P$ be a Sylow $p$-subgroup
of $N$. If a proper subgroup $K<L$ universally $p$-generates $L$,
then $\cosetlat(G,N)^{P\times K^{\diag}}=\emptyset$. \end{cor}
\begin{proof}
This follows immediately from Lemmas \ref{lem:UnivpgenFixedpts} and
\ref{lem:UnivProd}.\end{proof}
\begin{cor}
\label{cor:UnivCyclicorPp}Let $N,G,L,p$ and $K$ be as in Corollary~\ref{cor:EmptyfixProd}.
If $K$ is either cyclic or of prime-power order, then $\Delta\cosetlat(G,N)$
is not $\ff_{p}$-acyclic. \end{cor}
\begin{proof}
This follows directly from Corollary~\ref{cor:EmptyfixProd} and
Theorem~\ref{thm:smith}. 
\end{proof}
Our strategy is now clear. We go through the list of nonabelian simple
groups, as provided by the Classification. For each such group $L$,
we look for some $K<L$ that universally $2$-generates $L$ and is
either cyclic or of prime-power order. This strategy fails only when
$L=A_{7}$, in which case we use a slight extension of Corollary~\ref{cor:UnivCyclicorPp}.

Every nonabelian finite simple group is, up to isomorphism, an alternating
group $A_{n}$ with $n\geq5$, a group of Lie type, or one of twenty
six sporadic groups. Note that the small alternating groups $A_{5}\cong PSL_{2}(5)$,
$A_{6}\cong PSL_{2}(9)$ and $A_{8}\cong PSL_{4}(2)$ are all isomorphic
with simple groups of Lie type. (See for example \cite[Theorem 2.2.10]{Gorenstein/Lyons/Solomon:1998}.)
\begin{lem}
\label{lem:UniversalGeneration} If $L$ is simple and $L\not\cong A_{7}$,
then there is some $K<L$ that is either cyclic or of prime-power
order, and that universally $2$-generates $L$.

Indeed, the following claims hold. 
\begin{enumerate}
\item If $L=A_{n}$ with $n\geq9$ odd and $h\in L$ is an $n$-cycle, then
$\langle h\rangle$ universally $2$-generates $L$. 
\item If $L=A_{n}$ with $n\geq10$ even and $h\in L$ is an $(n-1)$-cycle,
then $\langle h\rangle$ universally $2$-generates $L$. 
\item If $L$ is a sporadic simple group or a simple group of Lie type,
then there is some odd prime $p$ such that a Sylow $p$-subgroup
of $L$ universally $2$-generates $L$. 
\end{enumerate}
\end{lem}
We will prove Claims~(1) and (2) in Section~\ref{sec:Alternating},
and Claim~(3) in Section~\ref{sec:LieSporadic}.

We examine $A_{7}$ in Section~\ref{sec:A7}, where we prove the
following result.
\begin{lem}
\label{lem:a7} If $G$ has a minimal normal subgroup $N\cong A_{7}^{t}$,
then $\Delta\cosetlat(G,N)$ is not $\ff_{2}$-acyclic. 
\end{lem}
Theorem~\ref{thm:MainIndStep} (and so Theorem~\ref{thm:Main})
follows from Lemma~\ref{lem:amns}, Corollary~\ref{cor:UnivCyclicorPp},
Lemma~\ref{lem:UniversalGeneration}, and Lemma~\ref{lem:a7}.

\section{\label{sec:Alternating}Alternating groups of high degree}

Here we prove Claims~(1) and (2) of Lemma~\ref{lem:UniversalGeneration}. 

Our proof involves the standard division of subgroups of $S_{n}$
into three classes. Such a subgroup $H$ is \textit{transitive} if
for each $i,j\in[n]$ there is some $x\in H$ such that $ix=j$, and
\textit{intransitive} otherwise. A transitive subgroup $H$ is \textit{imprimitive}
if there is some partition $\pi=\{\pi_{1},\ldots,\pi_{\ell}\}$ of
$[n]$ into subsets, such that $1<\ell<n$, and such that for each
$x\in H$ and each $i\in[\ell]$, there exists some $j\in[\ell]$
with $\pi_{i}x=\pi_{j}$. In this case, we say that $H$ \textit{stabilizes}
$\pi$. A subgroup $H$ is \textit{primitive} if $H$ is transitive
but not imprimitive. So, each subgroup of $S_{n}$ is intransitive,
imprimitive or primitive. We begin with a classical result of Jordan.
\begin{thm}[{Jordan \cite{Jordan:1875}. See also \cite[Example 3.3.1]{Dixon/Mortimer:1996}}]
\label{thm:jordan} If $n\geq9$, then every primitive subgroup of
$S_{n}$ containing an element with exactly $n-4$ fixed points contains
$A_{n}$. 
\end{thm}
When $n\geq4$, every Sylow $2$-subgroup of $A_{n}$ contains an
element with exactly $n-4$ fixed points, namely, the product of two
disjoint transpositions.
\begin{cor}
\label{cor:primitive} If $n\geq9$, then no primitive proper subgroup
of $A_{n}$ contains a Sylow $2$-subgroup of $A_{n}$. 
\end{cor}
Suppose that the transitive subgroup $H\leq S_{n}$ stabilizes the
partition $\pi=\{\pi_{1},\ldots,\pi_{\ell}\}$ of $[n]$, with $1<\ell<n$.
The transitivity of $H$ forces $|\pi_{i}|=|\pi_{j}|$ for all $i,j\in[\ell]$.
Each $\pi_{i}$ has size $d=n/\ell$. The full stabilizer $G_{\pi}$
of $\pi$ in $S_{n}$ (which contains $H$) is isomorphic with the
wreath product $S_{d}\wr S_{\ell}$, and thus has order $d!^{\ell}\ell!$.
Now $G_{\pi}\not\leq A_{n}$, as $G_{\pi}$ contains a transposition.
It follows that $G_{\pi}\cap A_{n}$ contains a Sylow $2$-subgroup
of $A_{n}$ if and only if $\frac{n!}{d!^{\ell}\ell!}$ is odd.
\begin{lem}
\label{lem:impodd} If $n$ is odd, then no imprimitive subgroup of
$A_{n}$ contains a Sylow $2$-subgroup of $A_{n}$. \end{lem}
\begin{proof}
By the preceding discussion, it suffices to show that $\frac{n!}{d!{}^{\ell}\ell!}$
is even whenever $d$ is a nontrivial proper divisor of $n$ and $\ell=n/d$.
Straightforward manipulations yield 
\[
\frac{n!}{d!{}^{\ell}\ell!}=\prod_{j=1}^{\ell}{jd-1 \choose d-1}.
\]
It suffices to show any term in the product on the right is even.
We calculate 
\[
{2d-1 \choose d-1}=\frac{2d-1}{d}{2d-2 \choose d-1}=\frac{2d-1}{d}\cdot2{2d-3 \choose d-1}.
\]
Since the divisor $d$ of $n$ is odd, the result follows.\end{proof}
\begin{lem}
\label{lem:impeven} If $n$ is even, then no imprimitive subgroup
of $A_{n}$ contains an $(n-1)$-cycle.\end{lem}
\begin{proof}
Assume for contradiction that the $(n-1)$-cycle $h\in A_{n}$ stabilizes
the partition $\pi=\{\pi_{1},\ldots,\pi_{\ell}\}$ with $1<\ell<n$.
Without loss of generality, the unique fixed point $j$ of $h$ lies
in $\pi_{1}$. Now $\pi_{1}h=\pi_{1}$. As $\langle h\rangle$ is
transitive on $[n]\setminus\{j\}$ and $|\pi_{1}|>1$, we obtain the
contradiction $[n]=\pi_{1}$. 
\end{proof}
No intransitive subgroup of $A_{n}$ contains an $n$-cycle. Thus
Claim~(1) follows from Corollary~\ref{cor:primitive} and Lemma~\ref{lem:impodd}.
To prove Claim~(2), it remains to show that when $n\geq10$ is even,
no intransitive subgroup of $A_{n}$ contains both a Sylow $2$-subgroup
of $A_{n}$ and an $(n-1)$-cycle. It suffices to show a Sylow $2$-subgroup
$P$ of $A_{n}$ contains a fixed-point-free element. Depending on
$n\mod4$, $P$ contains a conjugate of either 
\[
(1,2,3,4)(5,6)\cdots(n-1,n)\quad\mbox{or}\quad(1,2)(3,4)\cdots(n-1,n).
\]

\section{\label{sec:A7}The alternating group $A_{7}$ }

Here we prove Lemma~\ref{lem:a7}. 

The conclusion of Claim~(1) of Lemma~\ref{lem:UniversalGeneration}
does not hold for $A_{7}$. Indeed, $A_{7}$ has proper primitive
subgroups that contain both a 7-cycle and a Sylow $2$-subgroup of
$A_{7}$. (Such subgroups are isomorphic to $PGL_{3}(2)$, and are
embedded in $A_{7}$ through actions on points and lines in the Fano
plane.) As a result, we cannot apply Corollary~\ref{cor:EmptyfixProd}
or Corollary~\ref{cor:UnivCyclicorPp}. However, we can still apply
Smith Theory (as in Theorem~\ref{thm:smith}) when $G$ has a minimal
normal subgroup $N=A_{7}^{t}$.

We begin by collecting some information on $A_{7}$. All these facts
are straightforward to confirm, or can be verified with \cite{GAP4.4.12}
or \cite{Atlas}.
\begin{lem}
\label{lem:A7struc}The group $A_{7}$ has the following properties.
\begin{enumerate}
\item There exist conjugacy classes $\kone$ and $\ktwo$ of subgroups of
$A_{7}$ satisfying the following conditions.

\begin{enumerate}
\item If $K\in\kboth$, then $K\cong PGL_{3}(2)$, and so $[A_{7}:K]=15$.
\item A proper subgroup $K$ of $A_{7}$ contains both a Sylow $2$-subgroup
of $A_{7}$ and a 7-cycle if and only if $K\in\kboth$.
\item If $K\in\kboth$, $P$ is a Sylow $2$-subgroup of $K$, and $R$
is a Sylow $7$-subgroup of $K$, then $\left\langle P^{r}\,:\, r\in R\right\rangle =K$.
\end{enumerate}
\item There is an involution $\varphi\in\Aut(A_{7})$ satisfying the following
conditions.

\begin{enumerate}
\item The automorphism $\varphi$ normalizes both a Sylow $2$-subgroup
and a Sylow $7$-subgroup of $A_{7}$.
\item If $K\in\kone$ then $K^{\varphi}\in\ktwo$, and if $K\in\ktwo$ then
$K^{\varphi}\in\kone$.
\end{enumerate}

\noindent Indeed, the automorphism arising from conjugation by $(1,2)(3,4)(5,6)\in S_{7}$
has the desired properties.

\end{enumerate}
\end{lem}
The next result follows directly from Lemma~\ref{lem:A7struc} and
the universal property of direct products.
\begin{lem}
\label{lem:AutomOfN}Let $N$ be the direct product $\prod_{i=1}^{t}L_{i}$
with each $L_{i}\cong A_{7}$. There is an involution $\rho\in\Aut(N)$
such that the following claims hold for each $i\in[t]$.
\begin{enumerate}
\item The automorphism $\rho$ normalizes $L_{i}$.
\item The automorphism $\rho$ normalizes both a Sylow $2$-subgroup $P_{i}$
of $L_{i}$ and a Sylow $7$-subgroup $\left\langle h_{i}\right\rangle $
of $L_{i}$.
\item If $H\leq L_{i}$ is normalized by $\rho$ and contains both $P_{i}$
and $\left\langle h_{i}\right\rangle $, then $H=L_{i}$.
\end{enumerate}
Moreover, $\rho$ normalizes both $P:=P_{1}P_{2}\cdots P_{t}$ and
$K:=\left\langle (h_{1},h_{2},\dots,h_{t})\right\rangle $.\end{lem}
\begin{rem}
The automorphism $\rho$ may be realized concretely by embedding $N\cong A_{7}^{t}$
in $S_{7}^{t}$, setting $x:=(1,2)(3,4)(5,6)\in S_{7}$, and conjugating
by the element $(x,x,\dots,x)\in S_{7}^{t}$.
\end{rem}
The next lemma is a special case of a theorem of P.~Jin in \cite{Jin:2007}.
It also can be proved directly using standard facts from the cohomology
of groups (see \cite[Chapter IV]{Brown:1994}, particularly Corollary
IV.6.8 therein).

We write $\Inn(M)$ for the inner automorphism group of a group $M$
and $\Out(M)$ for the outer automorphism group $\Aut(M)/\Inn(M)$.
\begin{lem}[{See \cite[Corollary C]{Jin:2007}}]
\label{lem:ExtendAut} Let $N\normalineq G$ with $Z(N)=1$ and let
$\rho\in\Aut(N)$. If the coset $\Inn(N)\rho$ is in $Z(\Out(N))$,
then there exists some $\theta\in\Aut(G)$ such that
\begin{enumerate}
\item $\left|\theta\right|=\left|\rho\right|$,
\item $\theta$ normalizes $N$, and
\item the restriction $\theta_{N}$ of $\theta$ to $N$ is $\rho$.
\end{enumerate}
\end{lem}
Applying Lemma~\ref{lem:ExtendAut} to the involution $\rho$ described
in Lemma~\ref{lem:AutomOfN}, we obtain the following corollary.
\begin{cor}
\label{lem:A7toAutG}Suppose that the group $G$ has a normal subgroup
$N=\prod_{i=1}^{t}L_{i}$ with each $L_{i}\cong A_{7}$. If $\rho\in\Aut(N)$
is as in Lemma~\ref{lem:AutomOfN}, then there exists an involution
$\theta\in\Aut(G)$ such that $\theta$ normalizes $N$ and $\theta_{N}=\rho$.\end{cor}
\begin{proof}
Note first that $\Aut(A_{7})\cong S_{7}$ (see for example \cite[Section 8.2]{Dixon/Mortimer:1996}).
It follows that $\Aut(N)\cong S_{7}\wr S_{t}$ (this is \cite[Exercise 4.3.9]{Dixon/Mortimer:1996}).
As $Z(N)=1$, it follows in turn that $\Inn(N)\cong N$ and $\Out(N)\cong\mathbb{Z}_{2}\wr S_{t}$.
Therefore $\Out(N)$ has a central element $z$ of order $2$. 

Let $\varphi\in\Aut(N)$. The coset $\Inn(N)\varphi$ is equal to
$z$ if and only if the conditions
\begin{itemize}
\item [(a)] $\varphi$ normalizes $L_{i}$, and
\item [(b)] the restriction of $\varphi$ to $L_{i}$ is not in $\Inn(L_{i})$
\end{itemize}
are satisfied for each $i\in[t]$. The automorphism $\rho$ meets
conditions (a) and (b). The conclusion now follows from Lemma~\ref{lem:ExtendAut}.
\end{proof}
We are ready to complete the proof of Lemma~\ref{lem:a7}. Suppose
that the group $G$ has a minimal normal subgroup $N=\prod_{i=1}^{t}L_{i}$
with each $L_{i}\cong A_{7}$. The automorphism $\theta$ obtained
in Corollary~\ref{lem:A7toAutG} normalizes both groups $P,K$ described
in Lemma~\ref{lem:AutomOfN}. Using the componentwise action of $\theta$
on $P\times K$, we obtain the semidirect product 
\[
E:=(P\times K)\semidirect\left\langle \theta\right\rangle \leq(G\times G)\semidirect\Aut(G).
\]
 The group $E$ acts on $\cosetlat(G)$ as described in Section~\ref{sec:SmithTheory}.
Since $\theta$ normalizes $N$, this action restricts to $\cosetlat(G,N)$.
Since $\left|\theta\right|=2$, the group $E$ meets the conditions
of Theorem~\ref{thm:smith} (as discussed in Remark \ref{rem:SmithAppl}).
It thus suffices to show that $\cosetlat(G,N)^{E}=\emptyset$. 

Assume for contradiction that $Hx\in\cosetlat(G,N)^{E}$. Then $H$
contains $\left\langle P,K^{x^{-1}}\right\rangle $ and is normalized
by $\theta$. In particular, for each $i\in[t]$, the intersection
$H\cap L_{i}$ is normalized by $\theta$ and contains $P_{i}$. Moreover,
the projection of $H\cap N$ to $L_{i}$ (which contains $H\cap L_{i}$
as a normal subgroup) is normalized by $\theta$ and contains $\left\langle h_{i}\right\rangle $.
By Lemma~\ref{lem:AutomOfN}~(3), this projection is $L_{i}$. As
$A_{7}$ is simple and $H\cap L_{i}$ is nontrivial, it follows that
$H\cap L_{i}=L_{i}$. Therefore $N\leq H$. This is impossible, as
$H<G$ and $HN=G$.

\section{\label{sec:LieSporadic}Simple groups of Lie type and sporadic groups}

Here we prove Claim~(3) of Lemma~\ref{lem:UniversalGeneration}.
We refer the reader to \cite[Chapter 2]{Gorenstein/Lyons/Solomon:1998}
for an introduction to the finite simple groups of Lie type, with
original references. Each such group is determined by an irreducible
crystallographic root system $\Sigma$, a (possibly trivial) automorphism
$\sigma$ of the Dynkin diagram of $\Sigma$ and a finite field $\ff_{q}$
of order $q$. If $\sigma$ has order $d$, we say that the \textit{type}
of the associated simple group is $^{d}\Sigma(q)$, suppressing $d$
when $d=1$. Much of what we need has already been proved by Damian
and Lucchini in \cite[Section 4]{Damian/Lucchini:2007}. We summarize
their results as follows.
\begin{thm}[Damian and Lucchini \cite{Damian/Lucchini:2007}]
 \label{thm:DamianLucchini} If $L$ is a finite simple group of
Lie type or a sporadic simple group, then one of the following conditions
holds. 
\begin{enumerate}
\item [(a)] There is some cyclic subgroup $C\leq L$ of prime order that
universally $2$-generates $L$. 
\item [(b)] The group $L$ is of Lie type $B_{n}(q)$ $(n\geq3)$, $D_{n}(q)$
$(n\geq4)$, or $G_{2}(q)$, and $q$ is odd. 
\item [(c)] The group $L$ is of Lie type $A_{5}(2)$, $C_{3}(2)$, $D_{4}(2)$
or $^{2}A_{3}(2)$. 
\item [(d)] The group $L$ is the McLaughlin sporadic group $McL$. 
\end{enumerate}
\end{thm}
In fact, Damian and Lucchini give further restrictions on $L$ when
condition (b) of Theorem~\ref{thm:DamianLucchini} holds, but we
will not need these.

As every Sylow $p$-subgroup of a group $G$ contains a conjugate
of every element of order $p$ in $G$, it remains to examine the
groups listed in cases (b), (c) and (d) of Theorem~\ref{thm:DamianLucchini}. 

Suppose $L$ is a simple group of type $^{d}\Sigma(q)$ and $q$ is
a power of the prime $p$. We say that $L$ has \textit{characteristic}
$p$ and call a subgroup $M\leq L$ \textit{parabolic} if $M$ contains
the normalizer of some Sylow $p$-subgroup of $L$. The following
result is attributed to Tits by Seitz in \cite[(1.6)]{Seitz:1973}.
\begin{lem}[Tits]
\label{lem:tits} Let $L$ be a simple group of Lie type in characteristic
$p$, and let $P$ be a Sylow $p$-subgroup of $L$. Every maximal
subgroup of $L$ containing $P$ is parabolic. 
\end{lem}
The groups of types $A_{5}(2)$, $C_{3}(2)$, $D_{4}(2)$ and $^{2}A_{3}(2)$
are all classical groups. The orders of parabolic subgroups of classical
groups are known, and can be found in \cite{Kleidman/Liebeck:1990}.
It is straightforward to confirm that the index of each parabolic
subgroup of each of the four given groups is divisible by three. The
same holds for the index of each maximal subgroup of odd index in
$McL$, as can be confirmed by consulting \cite{Atlas}. We obtain
the following result.
\begin{lem}
\label{lem:cd} If the simple group $L$ is listed in case (c) or
case (d) of Theorem~\ref{thm:DamianLucchini}, then $L$ is universally
$2$-generated by any of its Sylow $3$-subgroups.
\end{lem}
It remains to handle case (b). Key to the work of Damian and Lucchini
in \cite{Damian/Lucchini:2007} is a result of Liebeck and Saxl in
\cite{Liebeck/Saxl:1985}, in which the authors describe all primitive
permutation groups of odd degree. (See also the paper \cite{Kantor:1987}
of Kantor, in particular Lemma~2.3 therein.) Such a description necessarily
includes a list of all pairs $(M,L)$ such that $L$ is a finite simple
group and $M$ is a maximal subgroup of odd index in $L$. Consulting
this list, we obtain the following result.
\begin{lem}[{Liebeck and Saxl \cite{Liebeck/Saxl:1985}; see also \cite[Lemma 2.3]{Kantor:1987}}]
\label{lem:oddpar} If $L$ is a simple group of Lie type in odd
characteristic and some parabolic subgroup of $L$ contains a Sylow
$2$-subgroup of $L$, then the type of $L$ is one of $A_{n}(q)$
or $E_{6}(q)$. 
\end{lem}
Consulting \cite[Theorem 2.2.10]{Gorenstein/Lyons/Solomon:1998},
we see that, assuming the lower bounds on $n$ listed in case (b),
there exists no isomorphism between a group of type $B_{n}(q)$, $D_{n}(q)$
or $G_{2}(q)$ with $q$ odd and a group of type $A_{n}(q)$ or $E_{6}(q)$.
Thus Lemmas~\ref{lem:tits} and \ref{lem:oddpar} together complete
our proof of Claim~(3).

\bibliographystyle{1_Users_russw_Documents_Research_mypapers_Coset_posets_are_noncontractible_hamsplain}
\bibliography{0_Users_russw_Documents_Research_mypapers_Coset_posets_are_noncontractible_Master}

\def\cprime{$'$}
\providecommand{\bysame}{\leavevmode\hbox to3em{\hrulefill}\thinspace}
\providecommand{\href}[2]{#2}
\begin{thebibliography}{10}

\bibitem{Abels/Holz:1993}
Herbert Abels and Stephan Holz, \emph{Higher generation by subgroups}, J.
  Algebra \textbf{160} (1993), no.~2, 310--341.

\bibitem{Aschbacher:2000}
Michael Aschbacher, \emph{Finite group theory}, second ed., Cambridge Studies
  in Advanced Mathematics, vol.~10, Cambridge University Press, Cambridge,
  2000.

\bibitem{Beck:2012}
Matthias Beck, \emph{Combinatorial reciprocity theorems}, Jahresber. Dtsch.
  Math.-Ver. \textbf{114} (2012), no.~1, 3--22, {arXiv:1201.2212}.

\bibitem{Bjorner:1995}
Anders Bj{\"o}rner, \emph{Topological methods}, Handbook of combinatorics,
  Vol.\ 1,\ 2, Elsevier, Amsterdam, 1995, pp.~1819--1872.

\bibitem{Boston:1996}
Nigel Boston, \emph{A probabilistic generalization of the {R}iemann zeta
  function}, Analytic number theory, {V}ol. 1 ({A}llerton {P}ark, {IL}, 1995),
  Progr. Math., vol. 138, Birkh\"auser Boston, Boston, MA, 1996, pp.~155--162.

\bibitem{Brown:1974}
Kenneth~S. Brown, \emph{Euler characteristics of discrete groups and
  {$G$}-spaces}, Invent. Math. \textbf{27} (1974), 229--264.

\bibitem{Brown:1975}
\bysame, \emph{Euler characteristics of groups: the {$p$}-fractional part},
  Invent. Math. \textbf{29} (1975), no.~1, 1--5.

\bibitem{Brown:1994}
\bysame, \emph{Cohomology of groups}, Graduate Texts in Mathematics, vol.~87,
  Springer-Verlag, New York, 1994, Corrected reprint of the 1982 original.

\bibitem{Brown:2000}
\bysame, \emph{The coset poset and probabilistic zeta function of a finite
  group}, J. Algebra \textbf{225} (2000), no.~2, 989--1012.

\bibitem{Damian/Lucchini:2007}
Erika Damian and Andrea Lucchini, \emph{The probabilistic zeta function of
  finite simple groups}, J. Algebra \textbf{313} (2007), no.~2, 957--971.

\bibitem{Damian/Lucchini:2014}
\bysame, \emph{Profinite groups in which the probabilistic zeta function
  coincides with the subgroup zeta function}, J. Algebra \textbf{402} (2014),
  92--119.

\bibitem{Damian/Lucchini/Morini:2004}
Erika Damian, Andrea Lucchini, and Fiorenza Morini, \emph{Some properties of
  the probabilistic zeta function on finite simple groups}, Pacific J. Math.
  \textbf{215} (2004), no.~1, 3--14.

\bibitem{Detomi/Lucchini:2003}
Eloisa Detomi and Andrea Lucchini, \emph{Recognizing soluble groups from their
  probabilistic zeta functions}, Bull. London Math. Soc. \textbf{35} (2003),
  no.~5, 659--664.

\bibitem{Dixon/Mortimer:1996}
John~D. Dixon and Brian Mortimer, \emph{Permutation groups}, Graduate Texts in
  Mathematics, vol. 163, Springer-Verlag, New York, 1996.

\bibitem{Dolfi/Guralnick/Herzog/Praeger:2012}
Silvio Dolfi, Robert~M. Guralnick, Marcel Herzog, and Cheryl~E. Praeger,
  \emph{A new solvability criterion for finite groups}, J. Lond. Math. Soc. (2)
  \textbf{85} (2012), no.~2, 269--281, {arXiv:1105.0475}.

\bibitem{GAP4.4.12}
The GAP~Group, \emph{{GAP -- Groups, Algorithms, and Programming, Version
  4.4.12}}, 2008.

\bibitem{Gaschutz:1959}
Wolfgang Gasch{\"u}tz, \emph{Die {E}ulersche {F}unktion endlicher aufl\"osbarer
  {G}ruppen}, Illinois J. Math. \textbf{3} (1959), 469--476.

\bibitem{Gorenstein/Lyons/Solomon:1998}
Daniel Gorenstein, Richard Lyons, and Ronald Solomon, \emph{The classification
  of the finite simple groups. {N}umber 3. {P}art {I}. {C}hapter {A}},
  Mathematical Surveys and Monographs, vol.~40, American Mathematical Society,
  Providence, RI, 1998, Almost simple $K$-groups.

\bibitem{Hall:1936}
Philip Hall, \emph{The {E}ulerian functions of a group}, Quart. J. Math.
  \textbf{7} (1936), 134--151.

\bibitem{Hatcher:2002}
Allen Hatcher, \emph{Algebraic topology}, Cambridge University Press,
  Cambridge, 2002, {http://www.math.cornell.edu/$\sim$hatcher/AT/ATpage.html}.

\bibitem{Jin:2007}
Ping Jin, \emph{Automorphisms of groups}, J. Algebra \textbf{312} (2007),
  no.~2, 562--569.

\bibitem{Jordan:1875}
Camille Jordan, \emph{Sur la limite du degr{\'e} des groupes primitifs qui
  contiennent une substitution donn{\'e}e}, J. Reine Angew. Math. \textbf{79}
  (1875), 248--258.

\bibitem{Kahn/Saks/Sturtevant:1984}
Jeff Kahn, Michael Saks, and Dean Sturtevant, \emph{A topological approach to
  evasiveness}, Combinatorica \textbf{4} (1984), no.~4, 297--306.

\bibitem{Kantor:1987}
William~M. Kantor, \emph{Primitive permutation groups of odd degree, and an
  application to finite projective planes}, J. Algebra \textbf{106} (1987),
  no.~1, 15--45.

\bibitem{Kleidman/Liebeck:1990}
Peter Kleidman and Martin Liebeck, \emph{The subgroup structure of the finite
  classical groups}, London Mathematical Society Lecture Note Series, vol. 129,
  Cambridge University Press, Cambridge, 1990.

\bibitem{Liebeck/Saxl:1985}
Martin~W. Liebeck and Jan Saxl, \emph{The primitive permutation groups of odd
  degree}, J. London Math. Soc. (2) \textbf{31} (1985), no.~2, 250--264.

\bibitem{Mann:1996}
Avinoam Mann, \emph{Positively finitely generated groups}, Forum Math.
  \textbf{8} (1996), no.~4, 429--459.

\bibitem{Munkres:1984}
James~R. Munkres, \emph{Elements of algebraic topology}, Addison-Wesley
  Publishing Company, Menlo Park, CA, 1984.

\bibitem{Oliver:1975}
Robert Oliver, \emph{Fixed-point sets of group actions on finite acyclic
  complexes}, Comment. Math. Helv. \textbf{50} (1975), 155--177.

\bibitem{Patassini:2009}
Massimiliano Patassini, \emph{The probabilistic zeta function of {${\rm
  PSL}(2,q)$}, of the {S}uzuki groups {${}^2B_2(q)$} and of the {R}ee groups
  {${}^2G_2(q)$}}, Pacific J. Math. \textbf{240} (2009), no.~1, 185--200.

\bibitem{Patassini:2011}
\bysame, \emph{On the (non-)contractibility of the order complex of the coset
  poset of a classical group}, J. Algebra \textbf{343} (2011), 37--77.

\bibitem{Patassini:2011b}
\bysame, \emph{Recognizing the characteristic of a simple group of {L}ie type
  from its probabilistic zeta function}, J. Algebra \textbf{332} (2011),
  480--499.

\bibitem{Patassini:2013b}
\bysame, \emph{On the irreducibility of the {D}irichlet polynomial of an
  alternating group}, Trans. Amer. Math. Soc. \textbf{365} (2013), no.~8,
  4041--4062.

\bibitem{Patassini:2013}
\bysame, \emph{On the (non-)contractibility of the order complex of the coset
  poset of an alternating group}, Rend. Semin. Mat. Univ. Padova \textbf{129}
  (2013), 35--46.

\bibitem{Quillen:1978}
Daniel Quillen, \emph{Homotopy properties of the poset of nontrivial
  {$p$}-subgroups of a group}, Adv. in Math. \textbf{28} (1978), no.~2,
  101--128.

\bibitem{Ramras:2005}
Daniel~A. Ramras, \emph{Connectivity of the coset poset and the subgroup poset
  of a group}, J. Group Theory \textbf{8} (2005), no.~6, 719--746,
  {arXiv:math/0210001}.

\bibitem{Seitz:1973}
Gary~M. Seitz, \emph{Flag-transitive subgroups of {C}hevalley groups}, Ann. of
  Math. (2) \textbf{97} (1973), 27--56.

\bibitem{Shareshian:1998}
John Shareshian, \emph{On the probabilistic zeta function for finite groups},
  J. Algebra \textbf{210} (1998), no.~2, 703--707.

\bibitem{Smith:1941}
P.~A. Smith, \emph{Fixed-point theorems for periodic transformations}, Amer. J.
  Math. \textbf{63} (1941), 1--8.

\bibitem{Stanley:1974a}
Richard~P. Stanley, \emph{Combinatorial reciprocity theorems}, Advances in
  Math. \textbf{14} (1974), 194--253.

\bibitem{Stanley:1974b}
\bysame, \emph{Combinatorial reciprocity theorems}, Combinatorics, {P}art 2:
  {G}raph theory; foundations, partitions and combinatorial geometry ({P}roc.
  {A}dv. {S}tudy {I}nst., {B}reukelen, 1974), Math. Centrum, Amsterdam, 1974,
  pp.~107--118. Math. Centre Tracts, No. 56.

\bibitem{Stanley:2012}
\bysame, \emph{Enumerative combinatorics. {V}olume 1}, second ed., Cambridge
  Studies in Advanced Mathematics, vol.~49, Cambridge University Press,
  Cambridge, 2012.

\bibitem{Torres-Giese:2012}
E.~Torres-Giese, \emph{Higher commutativity and nilpotency in finite groups},
  Bull. Lond. Math. Soc. \textbf{44} (2012), no.~6, 1259--1273.

\bibitem{Wachs:2007}
Michelle~L. Wachs, \emph{Poset topology: {T}ools and applications}, Geometric
  combinatorics, IAS/Park City Math. Ser., vol.~13, Amer. Math. Soc.,
  Providence, RI, 2007, {arXiv:math/0602226}, pp.~497--615.

\bibitem{Weisner:1935a}
Louis Weisner, \emph{Abstract theory of inversion of finite series}, Trans.
  Amer. Math. Soc. \textbf{38} (1935), no.~3, 474--484.

\bibitem{Atlas}
Robert Wilson~et al., \emph{Atlas of finite group representations},
  http://brauer.maths.qmul.ac.uk/Atlas/.

\bibitem{Woodroofe:2007}
Russ Woodroofe, \emph{Shelling the coset poset}, J. Combin. Theory Ser. A
  \textbf{114} (2007), no.~4, 733--746, {arXiv:math/0306346}.

\bibitem{Woodroofe:2009b}
\bysame, \emph{Cubical convex ear decompositions}, Electron. J. Combin.
  \textbf{16} (2009), no.~2, Special volume in honor of Anders Bjorner,
  Research Paper 17, approx. 33 pp., {arXiv:0709.2793}.

\end{thebibliography}

\end{document}